\documentclass[12pt, reqno]{amsart}

\usepackage{amssymb,latexsym,amsmath,amsfonts}
\usepackage{mathrsfs}
\usepackage{graphicx}
\usepackage[usenames]{color}
\usepackage{hyperref, cleveref}
\usepackage{comment}
\usepackage{enumitem}
\usepackage[normalem]{ulem}

\definecolor{DPurple}{rgb}{0.46,0.2,0.69}

\numberwithin{equation}{section}

\allowdisplaybreaks

\theoremstyle{definition}
\newtheorem{definition}{Definition}[section]

\theoremstyle{remark}

 \theoremstyle{plain}

\newtheorem{theorem}{Theorem}[section]
\newtheorem{lemma}[theorem]{Lemma}

\theoremstyle{definition}



\newcommand{\bea}{\begin{eqnarray*}}
\newcommand{\eea}{\end{eqnarray*}}

\numberwithin{equation}{section}

\date{}

\begin{document}
\title[Two remarks on transcendental shift-like maps on $\mathbb{C}^N$]{Two remarks on transcendental shift-like maps on $\mathbb{C}^N$} 
\author{Ramanpreet Kaur}
\address{Department of Mathematics, Institute of Mathematics and Applications, Bhubaneswar-751029, India}
\email{ramanpreet.ima@iomaorissa.ac.in, preetmaan444@gmail.com}
\keywords{transcendental entire function, transcendental shift-like map, topological entropy, wandering domains}
\subjclass[2020]{30D05, 37F10}

\begin{abstract}
In \cite{Bedford}, the dynamics of a particular polynomial diffeomorphism of $\mathbb{C}^N$, called a polynomial shift-like map of type $\nu$, has been studied as a higher dimensional analog of H\'enon maps. In this note, we prove that the Julia set of their transcendental counterpart is non-empty. In addition, an example of a transcendental shift-like map with an escaping wandering domain has been provided which, in fact, showcases a contrast with the dynamics of a polynomial shift map.
\end{abstract}
\maketitle

\vspace{-0.7cm}
\section{Introduction}\label{1}
Bedford and Pambuccian, in their paper titled,\textit{``Dynamics of shift-like polynomial diffeomorphisms on $\mathbb{C}^N$",} initiated the study of the dynamics of a polynomial shift map of type $\nu$ on $\mathbb{C}^N,$ where $N\geq 2$ \cite{Bedford}. In the spirit of H\'enon maps, they constructed filtration for a polynomial shift-like map of type $\nu$ which are indeed their higher dimensional analog. Loosely speaking, for $R>0$ (large enough) they constructed the sets $V_R^+$ (a forward invariant set in which the orbit of every point tends to infinity in forward time), $V_R^-$ (a set in which the forward orbit of any point can remain in it for a finite positive time), and $V_R$ (a compact set in which forward orbit of every point is bounded).
In this article, we study the dynamics of their transcendental counterpart, that is, a transcendental shift-like  map of type $\nu$ on $\mathbb{C}^N,$ where $1\leq\nu\leq N-1.$ In particular, we prove that its Julia set is non-empty. This indicates that it is worth studying the dynamics of a transcendental shift-like map. To study the dynamics of a transcendental diffeomorphism in higher dimensions, one may pursue the following three approaches to understand its dynamics or to develop techniques to study its dynamics.
\begin{enumerate}
\item Extending techniques from the transcendental dynamics for a map in several variables. Specifically, to extend the techniques from the dynamics of a transcendental H\'enon map.
\item Attempt to investigate from the perspective of a polynomial shift-like map in the setting of a transcendental shift-like map. 
\item Extend the known results for a transcendental entire function in one complex variable to their higher dimensional analogs.
\end{enumerate}
As  mentioned earlier, our primary focus is to show that the dynamics of a transcendental shift-like map is not linear, that is, its Julia set is non-empty. To achieve this, we shall follow the approaches $(1)$ and $(3)$, which in particular, is an amalgamation of ideas from transcendental dynamics in one variable and dynamics of a biholomorphism in several complex variables. Further, it is to be mentioned that the dynamics of transcendental shift-like maps is in contrast with that of their polynomial counterpart, for instance, existence of filtration for a polynomial shift-like map makes it possible for us to conclude that neither an escaping nor an oscillating wandering domain exists for a polynomial shift-like map. However, non-existence of filtration in transcendental setting arises a natural question: Does there exist a transcendental shift-like map with wandering domain? In Section $4$, we provide an example of a transcendental shift-like map with an escaping wandering domain. Also, it is still an open question  whether bounded wandering domains  exist for polynomials or not. Therefore, it is possible that a transcendental shift-like map may have wandering domains. In this article, we provide an example of a transcendental shift-like map with an escaping wandering domain, and hence showing a contrast with that of its polynomial counterpart. In a subsequent paper, we shall explore other aspects of transcendental shift-like maps, such as existence of periodic points, classification of recurrent Fatou components, etc.

To obtain that the Julia set of a transcendental shift-like map is non-empty, it is sufficient to show that its  topological entropy is infinite. For this, we shall use the method which Markus Wendt used in  \cite{Wendt} to show that any transcendental entire function $f$ has an infinite topological entropy. In fact, he divided the proof into two cases: when the sequence $\{f_n\}$, given by $f_n(z)=\frac{f(nz)}{n}$, is quasi-normal and when it is not. In quasi-normal case, it was obtained that each $f_n$ is a polynomial-like map of arbitrarily large degree, and in non-quasinormal case, by using \emph{Ahlfors Five Islands Theorem}, Wendt worked with arbitrary large number of disks with pairwise disjoint closures such that each of these disks contained a univalent preimage of all but atmost $2$ disks. 
It is to be mentioned that a similar approach was later used by Arosio et al. in order to achieve that the topological entropy of a transcendental H\'enon map is infinite \cite{Arosio3}.

\section{Preliminaries}
A diffeomorphism $F:\mathbb{C}^N\to\mathbb{C}^N$, with $N\geq 2$ is said to be shift-like if the orbit of a point $z\in\mathbb{C}^N$ under $F$ determines a bi-infinite sequence $\{\zeta_j\}_{j\in\mathbb{Z}}$ such that
\[F^k(z)=(\zeta_{k+1},\dots,\zeta_{k+N})\in\mathbb{C}^N. \]
Now, given a shift-like map $F$, it is said to have type $\nu$ for some $1\leq\nu\leq N-1$, if $F$ has the following form. For every $(z_1,z_2\dots,z_N)\in\mathbb{C}^N,$ 
\[F(z_1,z_2,\dots,z_N)=(z_2,z_3,\dots, z_N, f(z_{N-\nu+1})-az_1),\]
where $a\in\mathbb{C}^*$ and $f$ is a holomorphic function.  If  $f$ is a polynomial of degree $d$, then $F$ is said to be a polynomial shift-like map of type $\nu$ (see \cite{Bedford}). We say  that $F$ is a transcendental shift-like map of type $\nu$ if $f$ is a transcendental entire function.

\begin{definition}[Fatou set and Julia set]  \cite[Page 859]{Arosio}
A point $z\in\mathbb{C}^N$ belongs to the $\widehat{\mathbb{C}^N}$-Fatou set of $F$ if the family of iterates $\{F^n\}$ is $\widehat{\mathbb{C}^N}$-normal near $z$. A point $z\in\mathbb{C}^N$ belongs to the $\mathbb{P}^N$-Fatou set of $F$ if the family of iterates $\{F^n\}$ is $\mathbb{P}^N$-normal near $z$.
\end{definition}

\begin{definition}\cite[Page 857]{Arosio}
Let $X$ be a complex manifold. A family $\mathcal{F}\subset$ Hol$(X,\mathbb{P}^N)$ is $\mathbb{P}^N$-normal if for every sequence $\{f_n\}\subset\mathcal{F}$,
there exists a subsequence $\{f_{n_k}\}$ converging uniformly on compact subsets to some $f\in$ Hol$(X,\mathbb{P}^N)$. \\
A family $\mathcal{F}\subset$ Hol$(X,\mathbb{C}^N)$ is $\mathbb{C}^N$-normal if for every sequence $\{f_n\}\subset\mathcal{F}$ which is not
divergent on compact subsets, there exists a subsequence $\{f_{n_k}\}$ converging uniformly on compact subsets to some $f\in$ Hol$(X,\mathbb{C}^N)$.
\end{definition}

\begin{definition}\cite[Page 209]{Schiff}
Let $D$ be a domain in $\mathbb{C}$. A family $\mathcal{F}$ of holomorphic functions defined on $D$ is said to be quasi-normal if for every sequence $\{f_n\}\subseteq \mathcal{F}$, there exists a finite set $Q\subset D$ and a subsequence $\{f_{n_k}\}$ of $\{f_n\}$ which converges uniformly on compact subsets of $D\setminus Q.$
\end{definition}

\begin{definition}\cite{Adler, Canvos}
Let $f: X\to X$ be a continuous map of a metric space $(X, d)$ and let $K$ be a forward invariant compact subset of $X$. For a given $\epsilon >0$ and $n\in\mathbb{N},$ we say that a set $S\subseteq K$ is $(n,\epsilon)-$separated if for any $x, y\in S$ with $x\not= y$, we have 
\[\max_{1\leq j\leq n-1}d(f^{j}(x), f^j(y))\geq \epsilon.\]
Let $s(n, \epsilon, f,K)$ denote the maximal cardinality of an $(n,\epsilon)-$separated set. Then the topological entropy of $f$ over $K$ is given by 
\[h(f,K)=\lim_{\epsilon\to 0}\limsup_{n\to\infty}\frac{\log s(n, \epsilon, f,K)}{n}.\]
And the topological entropy of $f$ is given by 
\[h(f)=\sup\{h(f,K): K\subseteq X \text{ is a forward invariant compact subset}\}.\]
Similarly, we can define the topological entropy in terms of covering sets as follows.
For a given $\epsilon >0$ and $n\in\mathbb{N},$ we say that a set $C\subseteq K$ is $(n,\epsilon)-$covering if for every $x\in K$, there exists a $y\in C$ such that
\[\max_{1\leq j\leq n-1}d(f^{j}(x), f^j(y))< \epsilon.\]
Let $c(n,\epsilon, f,K)$ denote the minimal cardinality of an $(n,\epsilon)-$covering set. Then the topological entropy of $f$ over $K$ is given by 
\[h(f,K)=\lim_{\epsilon\to 0}\limsup_{n\to\infty}\frac{\log c(n, \epsilon, f,K)}{n}.\]
And the topological entropy of $f$ is given by 
\[h(f)=\sup\{h(f,K): K\subseteq X \text{ is a forward invariant compact subset}\}.\]
\end{definition}

Recall that a transcendental shift-like map of type $\nu$ is a diffeomorphism of $\mathbb{C}^N$, given by
\begin{equation}\label{3}
F(z_1,z_2,\dots,z_N)=(z_2,z_3,\dots, z_N, f(z_{N-\nu+1})-az_1),
\end{equation}
where $f$ is a transcendental entire function and $a\in\mathbb{C}^*.$ Consider a sequence of rescaled maps,  $f_n(z)=\frac{f(nz)}{n}$ for every $z\in\mathbb{C}.$ Notice that for every $n\geq 1$, $f$ and $f_n$ are topologically conjugate via the map $z\to nz$. This gives that for every $n\in\mathbb{N}$, the maps $F_n$ given by $F_n(z_1,z_2,\dots, z_N)=(z_2,z_3,\dots, f_n(z_{N-\nu+1})-az_1)$ and $F$ are topologically conjugate to each other via the map $(z_1,z_2,\dots, z_N)\mapsto (nz_1,nz_2,\dots,n z_N)$. Hence the maps $F_n$ and $F$ have the same topological entropy for every $n\in\mathbb{N}$. In this direction, we consider two cases: when the sequence $\{f_n\}$ is quasi-normal on $\mathbb{C}$ and when the sequence $\{f_n\}$ is not quasi-normal.
Now, let us recall the following lemmas which we shall use in the case when the sequence $\{f_n\}$ is quasi-normal \cite{Arosio3}.\\
\indent Let $\{f_{n_k}\}$ be  a subsequence of $\{f_n\}$ and $Q$ be a finite set such that $\{f_{n_k}\}$ converges uniformly on compact subsets of $\mathbb{C}\setminus Q.$ Also, for any $s\in (0, \infty)$, denote the Euclidean disk of radius $s$ centred at origin by $\mathbb{D}_s$.

\begin{lemma}\label{2.3}
The set $Q$ contains the origin, and there exists $s\in (0,1)$ such  that the sequence $\{f_{n_k}\}$  diverges uniformly to infinity  on compact subsets of $\mathbb{D}_s\setminus\{0\}.$
\end{lemma}
\begin{lemma}\label{2.6}
Let $s, \{f_{n_k}\}$ be as in Lemma \ref{2.3}. Further, if $r\in (0,s), R>0$, and $m\in\mathbb{N}$, then there exists $k_0\geq 1$ such that for any $k>k_0$, the following holds.
\begin{enumerate}
\item $|f_{n_k}(z)|>R$ for every $z\in\partial\mathbb{D}_r.$
\item The winding number of the curve $f_{n_k}(\partial\mathbb{D}_r)$ around the origin is larger than or equal to $m$.
\end{enumerate}
\end{lemma}

Prior to recalling results in the direction of non-quasinormal case, let us first introduce some notations.
\begin{itemize}
\item[(i)] Denote by $\{f_{n_h}\}$ a subsequence of $\{f_n\}$ and $Q_1=\{x_j: j\in\mathbb{N}\}$ an infinite set such that no subsequence of $\{f_{n_h}\}$ converges uniformly in any neighbourhood of any $x_j$ \cite[Proposition 2.9]{Arosio3}. 
\item[(ii)] Fix $k\geq 1$, and let $R$ be a positive number such that $\overline{\mathbb{D}_R(x_i)}\cap \overline{\mathbb{D}_R(x_j)}=\emptyset$ for every $1\leq i\not=j\leq k.$
\item[(iii)] Choose $r\in (0,R)$ such that $|a|r<R-r.$
\item[(iv)] For any $n_h\in\mathbb{N}$ and for any $i,l\in\{1,2,\dots,k\}$, denote
\begin{align*}\label{iv}
J(i,l)=&\{j\in\{1,2,\dots,k\}: \mathbb{D}_R(x_j+ax_l)\text{ admits a biholomorphic preimage}\\
&\text{ under }f_{n_h} \text{ in }
 \mathbb{D}_r(x_i)\}.
\end{align*} 
\end{itemize}

For the non-quasinormal case, we shall make use of the following lemma.
\begin{lemma}\cite[Lemma 3.10]{Arosio3}\label{3.1}
There exists $n_h$ such that the cardinality of the set $J(i,l)$ is atleast $k-2$ for every $i,l\in\{1,2,\dots,k\}.$
\end{lemma}

\section{Non-emptiness of the Julia set }
As we mentioned in Section \ref{1}, to show that a transcendental shift-like map of type $\nu$ has non-empty Julia set, it is enough to show that its topological entropy is infinite. In this section, we prove the following.
\begin{theorem}
If $F$ is a transcendental shift-like map of type $\nu$ on $\mathbb{C}^N$, then its topological entropy is infinite.
\end{theorem}
\begin{proof}
We shall divide our proof into two cases.\\
\textbf{Case 1:} When the sequence $\{f_n\}$ is quasi-normal. \\
Here, we shall first prove that if $f$ is a holomorphic  function defined in a neighbourhood of $\overline{\mathbb{D}_r}$ and $a\in\mathbb{C}^*$ is  such that $|f(z)|>(|a|+1)r$ whenever $|z|=r$, and also if the winding number of the curve $f(\partial\mathbb{D}_r)$ around the origin is $d\geq 1$, then the topological entropy of the map $F^{\nu}$, where $F$ is map as defined in \Cref{3}, is at least $\log d.$
To obtain this claim, we shall use a technique as used by Dujardin in \cite[Theorem 3.1]{Dujardin}. Consider the following notations.
\begin{align*}
\Delta&=\underbrace{\mathbb{D}_r\times\mathbb{D}_r\times\dots\times\mathbb{D}_r}_{\substack{N\text{-times}}}=\prod_{i=1}^N \mathbb{D}_r.\\
\partial^{\nu}_+\Delta &=\{(z_1,z_2,\dots,z_N)\in\mathbb{C}^N: |z_i|=r \text{ for some } N-\nu+1\leq i\leq 	N\}, \text{ and } \\
\partial^{N-\nu}_-\Delta&=\{(z_1,z_2,\dots,z_N)\in\mathbb{C}^N: |z_i|=r \text{ for some } 1\leq i\leq 	N-\nu\}.
\end{align*}

Take $K=\bigcap_{n\in\mathbb{N}}F^{-n\nu}\left(\overline{\Delta}\right)$. Let $c(n,\epsilon,K)$ be the minimal cardinality of $(n,\epsilon)-$spanning sets in $K$. For every $n\in\mathbb{N}$, denote the set $\bigcap_{k=1}^{n} F^{-k\nu}\left(\overline{\Delta}\right)$ by $D_n$. We first claim that 
\[\lim_{\epsilon\to 0}\limsup_{n\to\infty}\frac{1}{n}\log c(n,\epsilon,K)=\lim_{\epsilon\to 0}\limsup_{n\to\infty}\frac{1}{n}\log c(n,\epsilon,D_n).\]
Since $K\subseteq D_n$ for every $n\geq 1$, we have 
\[\lim_{\epsilon\to 0}\limsup_{n\to\infty}\frac{1}{n}\log c(n,\epsilon,K)\leq\lim_{\epsilon\to 0}\limsup_{n\to\infty}\frac{1}{n}\log c(n,\epsilon,D_n).\]
For the other inequality, we shall proceed as follows.\\
Define a relation $\sim$ on $\overline{\Delta}$ as: identify every point on $\partial^{\nu}_+\Delta$ with a point $(r,r,\dots,r)\in \overline{\Delta}$, and every other point of $\overline{\Delta}$ with itself. Therefore, we get a quotient map $\pi: \overline{\Delta}\to\overline{\Delta}/{\sim}$ defined by 
$\pi(z)=(r,r,\dots,r)$ for every $z\in \partial^{\nu}_+\Delta$, and $\pi(z)=z$ for every $z\not\in \partial^{\nu}_+\Delta.$ Now, we can induce a metric $\widetilde{d}$ on $\overline{\Delta}/{\sim}$ as follows.
\begin{align*}
\widetilde{d}(\pi(z_1,z_2,\dots,z_N), \pi(w_1,w_2,\dots,w_N))=&\min\{d((z_1,z_2,\dots,z_N), (w_1,w_2,\dots,w_N)), \\
&d((z_1,z_2,\dots,z_N), \partial^{\nu}_+\Delta), d((w_1,w_2,\dots,w_N), \partial^{\nu}_+\Delta)\},\\
\widetilde{d}(\pi(z_1,z_2,\dots,z_N), (r,r,\dots,r))=&\, d((z_1,z_2,\dots,z_N), \partial^{\nu}_+\Delta).
\end{align*}
Further, we define a map $\widetilde{F}:\overline{\Delta}/{\sim}\to\overline{\Delta}/{\sim} $ as $\widetilde{F}(\pi(z_1,z_2\dots,z_N))=\pi(F^{\nu}(z_1,z_2,\dots,z_N))$ if $F^{\nu}(z_1,z_2,\dots,z_N)\in\Delta$ and $\widetilde{F}(\pi(z_1,z_2,\dots,z_N))=(r,r,\dots,r)$ if $F^{\nu}(z_1,z_2,\dots,z_N)\not\in \Delta.$ 
The induced metric $\widetilde{d}$ on $\overline{\Delta}/{\sim}$ gives us that $\widetilde{F}$ is continuous. Also, by definition, we get $\widetilde{F}|_{\pi(K)}=F^{\nu}|_{K}.$ As $d$ and $\widetilde{d}$ coincide in some neighbourhood of $K$,  we have the topological entropies of $\widetilde{F}$ and $F^{\nu}$ are same on $\pi(K)$ and $K$, respectively. Further, as $\pi$ is an injective map in a neighbourhood of $K$, we get 
\begin{equation}\label{eq1}
h(F)=h(\widetilde{F})\geq \lim_{\epsilon\to 0}\limsup_{n\to\infty} \frac{1}{n}c(n,\epsilon,D_n).
\end{equation}
This, in particular gives us our claim. \\
Now, for every $N-\nu+1\leq i\leq N$, we have $\pi_i\circ F^{\nu}$ is a proper map of degree $d$. Therefore, for every horizontal line $L$ (in any direction parallel to the $j$-th coordinate, where $N-\nu+1\leq j\leq N$) we have
\begin{equation}\label{eq2}
\limsup_{n\to\infty} \frac{1}{n}\log \text{volume}(F^{n\nu}(L\cap D_n))=\log d.
\end{equation}
Lastly, by Yomdin's result \cite[Page no. 290-292]{Yomdin}, we get that
\[ \lim_{\epsilon\to 0}\limsup_{n\to\infty} \frac{1}{n}c(n,\epsilon,D_n)\geq \limsup_{n\to\infty} \frac{1}{n}\log \text{vol}(F^{n\nu}(L\cap D_n)). \]
Therefore, using \Cref{eq1} and \Cref{eq2}, we get  $h(F^{\nu})\geq \log d.$ \\
Now, from Lemma \ref{2.6}, we get for every $k>k_0$, $F_{n_k}^{\nu}$ satisfies all the above conditions.  Therefore, we get that the topological entropy of $F_{n_k}^{\nu}$ is at least $\log m$, and hence the same holds for $F^{\nu}$. Since $m$ is arbitrary, we get that the topological entropy of $F$ is infinite.

\noindent \textbf{Case 2:} When the sequence $\{f_n\}$ is not quasi-normal.\\
 In this case, we shall only prove the result for $\nu=N-\nu.$ The other cases follow in a similar manner.  Firstly, we introduce the following definition and then follow the steps given below.
\begin{definition}
Let $i_1,i_2,\dots, i_N\in\{1, 2, \dots,k\}$. We say that a domain $D$ is a $(i_1,i_2,\dots, i_N)$\\$-$holomorphic polydisk if it can be parametrized as follows.
\begin{align*}
D&=\{(w_{1}(z_{N-\nu+1}),w_2(z_{N-\nu+2}),\dots,w_{\nu}(z_N), z_{N-\nu+1}, z_{N-\nu+2},\dots, z_N)\in\mathbb{C}^N: z_{j}\in\mathbb{D}_r(x_{i_j})\\
& \text{ for every }N-\nu+1\leq j\leq N \text{ and } w_l \text{ is a holomorphic  function from }\mathbb{D}_r(x_{N-\nu+l}) \text{ to }\\
& \,\mathbb{D}_r(x_{i_l}) \text{ for every }1\leq l\leq \nu\}.
\end{align*}
\end{definition}
\noindent Step 1: 
If $i_1, i_2,\dots, i_N\in\{1,2,\dots,k\}$, then for every $j_1\in J(i_{N-\nu+1}, i_1), j_2\in J(i_{N-\nu+2}, i_2),$\\
$\dots,j_{\nu}\in J(i_{N}, i_{\nu})$ and for every $(i_1, i_2,\dots, i_N)-$holomorphic polydisk $D$, we observe that there exists a holomorphic polydisk $V\subseteq D$ such that $F_{n_h}^{\nu}(V)$ is a $(i_{N-\nu+1}, i_{N-\nu+2},\dots,i_N, j_1, $\\
$j_2,\dots, j_{\nu})-$holomorphic polydisk.\\
\vspace{0.2cm}
For this, firstly note that, by definition, for every $(z_1,z_2,\dots, z_{N})\in\mathbb{C}^N$ and for every $n_h\in\mathbb{N}$ (existence is guaranteed by Lemma \ref{3.1}), we have 
\\
\begin{align*}
F^{\nu}_{n_h}(z_1, z_2,\dots, z_N)=&(z_{\nu+1}, z_{\nu+2}, \dots, z_N, f_{n_h}(z_{N-\nu+1})-az_1, f_{n_h}(z_{N-\nu+2})-az_2,\dots, f_{n_h}(z_{N-1})\\
&-az_{\nu-1}, f_{n_h}(z_{N})-az_{\nu}).
\end{align*}
Therefore, if we take any point $(z_1, z_2,\dots,z_N)$ in a  $(i_1, i_2,\dots, i_N)-$holomorphic polydisk $D$, then the first $N-\nu$ components of its image under $F^{\nu}_{n_h}$ lies in $\mathbb{D}_r(x_{i_{\nu+1}}), \mathbb{D}_r(x_{i_{\nu+2}}), \dots,$\\
$\mathbb{D}_r(x_{i_{N}})$, respectively. For the remaining $\nu$ components, we shall make the following observation.
For every $1\leq l\leq \nu,$ we have $j_{l}\in J(i_{N-\nu+l}, i_l)$. Therefore, by Lemma \ref{3.1}, $\mathbb{D}_R(x_{j_{l}}+ax_{i_{l}})$ admits a biholomorphic preimage under $f_{n_h}$ in $\mathbb{D}_r(x_{i_{N-\nu+l}})$, that is, there exists a $W_l\subset \mathbb{D}_r(x_{i_{N-\nu+l}})$ such that $f_{n_h}-ax_{i_l}$ is a biholomorphism between $W_l$ and $\mathbb{D}_R(x_{j_l})$. We can further assume that (by shrinking $R$) $f_{n_h}-ax_{i_l}$ is a homeomorphism between $\overline{W_l}$ and  $\overline{\mathbb{D}_R(x_{j_l})}$. Now, for any $z_{N-\nu+l}\in\partial W_l$, we have 
\[|(f_{n_h}(z_{N-\nu+l})-aw_l(z_{N-\nu+l}))-(f_{n_h}(z_{N-\nu+l})-ax_{i_l})|=|a||w_l(z_{N-\nu+l})-x_{i_l}|<|a|r<R-r.\]
Since, $f_{n_h}-ax_{i_l}$ is a biholomorphism, by Rouch\'e's Theorem, we get that $f_{n_h}-aw_l$ is also an injection on $W_l$. Hence, for every $1\leq l\leq \nu$, if we set $U_l=(f_{n_h}-aw_l)^{-1}(\mathbb{D}_r(x_{j_l}))$, then we get $f_{n_h}-aw_l$ a biholomorphism between $U_l$ and $\mathbb{D}_r(x_{j_l})$. Hence, if we take $V=D\cap (\mathbb{C}\times\dots\times \mathbb{C}\times U_1\times U_2\times\dots\times U_{\nu})$, then we get $F^{\nu}_{n_h}(V)$ is a $(i_{N-\nu+1}, i_{N-\nu+2},\dots,i_N, j_1, j_2,\dots, j_{\nu})-$ holomorphic polydisk.\\
Step 2:  Define a set 
\[K=\bigcup_{1\leq i_1,i_2,\dots,i_N\leq k}\overline{\mathbb{D}_r(x_{i_1})}\times\overline{\mathbb{D}_r(x_{i_2})}\times\dots\times\overline{\mathbb{D}_r(x_{i_N})} \,\text{ and }\,L=\bigcap_{m\geq 0}F_{n_h}^{-m\nu}(K)\]
It can be easily seen that $K$ is compact and $L$ is forward $F_{n_h}^{\nu}-$invariant. Now, we associate every element of $L$ with a sequence in the set $\{1,2,\dots,k\}^{\mathbb{N}}$. For this, we say that a sequence $(l_n)_{n\geq 1}\in\{1,2,\dots,k\}^{\mathbb{N}}$ is admissible if $l_{N+m}=J(l_{N-\nu+m}, l_m)$ for every $m\geq 1$. A finite word is admissible if it is the start of an infinite admissible sequence. Now, for every admissible sequence $(l_n)_{n\geq 1}$, by Step 1, there exists a point $v\in L$ such that $F_{n_h}^{m\nu}(v)$ lies in a $(l_{m\nu+1}, l_{m\nu+2},\dots,l_{N+(m-1)\nu}, l_{N+{(m-1)\nu}+1},\dots, l_{N+{(m-1)\nu+\nu}})$ domain for every $m\geq 1$. Further, for every $ 1\leq m\leq N$, there are at least $k^N$ words of length $m$ and for $m>N$, there are at least $k^N(k-2)^{m-N}$ words of length $m$. Therefore, $L$ has at least $(k-2)^{m}$ elements with distinct symbolic representation, and these elements will be $(m,\epsilon)-$ separated if 
\[\epsilon<\min_{1\leq r,s\leq N}\text{dist}(\overline{\mathbb{D}}_r(x_{i_r}), \overline{\mathbb{D}}_r(x_{i_s}))\]
Therefore, on $L$, $F_{n_h}^{\nu}$ has topological entropy at least $\log (k-2)$. Since $k$ is arbitrary, we get $h(F_{n_h}^{\nu})=\infty$, and this implies that a transcendental shift-like map of type $\nu$ has infinite entropy.
\end{proof}

\section{Wandering domains: an example}
In this section, we shall provide a contrast between the dynamics of polynomial shift-like maps and their transcendental counterparts. In fact, we provide an example of transcendental shift-like map with an escaping wandering domain. Recall that a map $F$ defined on $\mathbb{C}^N$ has an escaping wandering domain $U$ if the orbit of every point of $U$ under $F$ escapes to infinity \cite{Arosio}. Note that, escaping wandering domains do not exist for a polynomial shift-like map due to the existence of filtrations \cite{Bedford}. Specifically, we prove the following. 
\begin{theorem}
There exists a transcendental shift-like map on $\mathbb{C}^3$ with escaping wandering domains.
\end{theorem}
\begin{proof}
To obtain a desirable map, we shall proceed with the following steps.\\
\noindent Step 1: Consider $f(z)=z+\sin(2\pi z)+1-\sqrt{1-\frac{1}{4\pi^2}}$. 
In \cite{Arosio}, it has been shown that the 
critical points of $f$ 
are $c_{n,+}=\alpha+n$ 
and $c_{n,-}=-\alpha+n$, 
for $n\in\mathbb{Z}$ and with $\alpha\in \left(\frac{1}{4},\frac{1}{2}\right)$ 
to be chosen such that $\sin(2\pi(\alpha+n))=\sqrt{1-\frac{1}{4\pi^2}}.$ 
We shall consider only $c_{n,+}, n\in\mathbb{Z}$. Also, note that 
\begin{align*}
f(c_{n,+})&=\alpha+n+\sin(2\pi(\alpha+n))+1-\sqrt{1-\frac{1}{4\pi^2}}\\
&=\alpha+n+\sqrt{1-\frac{1}{4\pi^2}}+1-\sqrt{1-\frac{1}{4\pi^2}}\\
&=\alpha+n+1\\
&=c_{n+1,+}.
\end{align*}
It means that $f$ acts as a translation on $c_{n,+},$ for $n\in\mathbb{Z}$.

\noindent Step 2: Consider $\widetilde{f}(z)=\phi\circ f\circ \phi^{-1}(z)$, 
where $\phi(z)=z-\alpha$. 
Therefore, $\widetilde{f}(z)=z+\sin (2\pi(z+\alpha))+1-\sqrt{1-\frac{1}{4\pi^2}}$. 
Now, let $z_n=n,$ for $n\in\mathbb{Z}$. 
We can easily observe that $\widetilde{f}(n)=n+1,$ 
for  $n\in\mathbb{Z}$, 
and each $z_n$ is a critical point for $\tilde{f}$.

\noindent Step 3: $\widetilde{f}$ commutes with translation by $1$, 
that is, $\widetilde{f}(z+1)=\widetilde{f}(z)+1,$ 
for every $z\in\mathbb{C}.$ 
Because, $\widetilde{f}(z+1)=z+1+\sin (2\pi(z+\alpha))+1-\sqrt{1-\frac{1}{4\pi^2}}=\widetilde{f}(z)+1$.

\noindent Step 4: Define the transcendental shift map of type $\nu=1$ as follows
\[F(z_1,z_2, z_3)=\left(z_2,z_3,\widetilde{f}(z_{3})+a(z_{3}-1)-a z_1+a\right),\]
where $\widetilde{f}$ is as constructed in Step 2, and $a\in(0,1)$. 

\noindent Step 5: $F$ commutes with translation by $(1,1,1)$, 
that is, $F(z_1+1,z_2+1, z_3+1)=F(z_1,z_2, z_3)+(1,1,1)$, 
for every $(z_1,z_2, z_3)\in\mathbb{C}^{3}$. 
For this, consider
\begin{align*}
F(z_1+1,z_2+1, z_3+1)& = (z_2+1,z_3+1,\widetilde{f}(z_3+1)+az_3-a(z_1+1)+a)\\
&= (z_2+1,z_3+1,\widetilde{f}(z_3)+1+a(z_3-1)-az_1+a)\\
& = (z_2,z_3,\widetilde{f}(z_3)+a(z_3-1)-az_1+a)+(1,1,1)\\
& = F(z_1,z_2,z_3)+(1,1,1).
\end{align*}
Also, observe that 
for every $n\in\mathbb{Z},$
\[ F(n-1,n,n+1)=(n,n+1,n+2).\] That is, $F$ acts 
as a translation by $(1,1,\dots,1)$ 
on $q_n=(n-1,n, n+2),$ 
for every $n\in\mathbb{Z}$. Now, define a map $G$ as $G(z_1,z_2,z_3)=F(z_1,z_2,z_3)-(1,1,1)$. Therefore, if we denote translation map on $\mathbb{C}^3$ by $T_{(1,1,1)}$, then $F=G\circ T_{(1,1,1)}$. Also, we get that each $q_n$ is an attracting fixed point of $G$ since $a\in (0,1)$. 
Let $B_n$ be an immediate basin of attraction for $G$ corresponding 
to $q_n,$ for $ n\in\mathbb{Z}.$
It can be easily seen  
that $F(B_n) = B_{n+1},$ for $ n\in\mathbb{Z}.$ 
We shall show that 
each $B_n$ is also a Fatou component for $F$. If it is true, then we get $B_1$ is an escaping wandering domain for $F$.
To establish that, 
we shall follow the approach used by Arosio et al. in \cite[Theorem 6.3]{Arosio}. In fact, the following result is a generalization of Theorem $6.3$ proved in \cite{Arosio}.
To do so, we shall implement the following approach.\\
Let $U$ be an Fatou component of $G$ corresponding to an attracting fixed point $p$. We shall prove that  $U$ is also a Fatou component of $F$.  
 We first show that $U\subseteq V,$ where $V$ is 
a Fatou component of $F$. Then, we claim 
that every boundary point of $U$ lies in the 
Julia set of $F$, which in particular shows that $U=V$.

Firstly, observe that as $p$ is a fixed point 
for $G$, we have $\{T^n_{(1,1,1)}(p): n\in\mathbb{Z}_+\}$ 
is an $F$-orbit of $p$ and it converges to the 
point $[1:1:1:0]$ on the line at infinity. 
Since, $F$ commutes with a translation 
by $(1,1,1),$ we have $F^{ n}=T^n\circ G^n,$ for $ n\in\mathbb{N}$. It means that the $F-$orbit 
of every point in $U$ converges to $[1:1:1:0]$. Hence $U\subseteq V,$ where $V$ is a Fatou component of $F$.

Now, for simplicity, assume that $p=(0,0,0)$. 
For $q\in\partial U$, there exists a 
neighbourhood $U_1$ of the origin such that orbit of the  point $q$ 
under $G$ does not intersect $U_1$, that is, 
there exists $\mu_1\in(0,1)$ such that
\[\Vert G^n(q)\Vert>\sum_{k=0}^{\infty}\mu_1^k \,\text{ for every }n\in\mathbb{N}.\]
Further, if $\pi_n$ denotes the projection of $G^n(q)$ onto a complex line through origin, then we have 
\[ |\pi_n(G^n(q))|>\sum_{k=0}^{\infty}\mu_1^k \,\text{ for every }n\in\mathbb{N}.\]
Now, consider $\phi:\mathbb{C}\to\mathbb{C}^3$, a complex embedding such that $q\in\phi(\mathbb{D})$ and $\phi(0)\in U$. Also, define a sequence of entire functions, $\phi_n:\mathbb{C}\to\mathbb{C}$ as $\phi_n(z)=\pi_n\circ G^n\circ\phi(z)$. For this sequence, we first make the following observations.
\begin{enumerate}
\item Let $r>0$ be such that $\phi(D(0,r))$ be compactly contained in $U$. Since $U$ is an attracting Fatou component for $G$, we get the existence of  $C_1>0$ and $\lambda_1\in(0,1)$ such that 
\[\Vert\phi_n\Vert_{D(0,r)}<C_1\lambda_1^n.\]
Further, if we instead define
\[\phi_n = \pi_n \circ G^{n+m}\circ\phi\]
for some large integer $m\in\mathbb{N}$ we may assume that
$\Vert\phi_n\Vert_{D(0,r)}<\lambda_1^n.$
Now, by defining  $\xi_n(z)= \phi_n(r · z)$, we obtain a sequence of maps $\xi_n:\mathbb{D}\to\mathbb{C}$. Therefore, $|\xi_n(z)|<\lambda_1^n$. For this sequence of functions, we first claim that if $\beta\in(0,1)$, then $|\xi_n^{(k)}(0)|<k!\beta^k$ for $k<\frac{n}{\log_{\lambda_1}(\beta)}$.
Using Cauchy estimates, we have $|\xi_n^{(k)}(0)|<k!\lambda_1^n<k!\beta^k$ for $k<\frac{n}{\log_{\lambda_1}(\beta)}$ (by the given condition on $k$, we have $\lambda_1^n<\beta^k$). 

Now, we choose $\beta=\mu_1r$. Note that, we need $\beta\in (0,1)$. For this, we modify our $\mu_1$ accordingly. Using the above claim, we have \[|\xi_n^{(k)}(0)|<k!(\mu_1 r)^k \text{ for } k<\frac{n}{\log_{\lambda_1}(\mu_1 r)}.\]
This further implies that 
\[|\phi_n^{(k)}(0)|<k!\mu_1^k \text{ for } k<\frac{n}{\log_{\lambda_1}(\mu_1 r)}.\]
\item By the choice of $\phi$, we have $q\in\phi(\mathbb{D})$. Let $\zeta=\phi^{-1}(q)\in\mathbb{D}$. Using the definition of $\phi_n$, we have $\displaystyle |\phi_n(\zeta)|\geq \sum_{k=0}^{\infty}\mu_1^k$ for every $n\in\mathbb{N}.$ 
We have the following as an immediate consequence. If $\displaystyle \phi_n(z)=\sum_{k=0}^{\infty}a_k z^k$ is a power series expansion of $\phi_n$ around the origin, then 
there exists at least one $k\in\mathbb{N}$, say $k_0$, such that $|a_{k_0}|\geq \mu_1^{k_0}$.\\
Combining this with observation $(1)$, we have 
\[|\phi_n^{(k_0)}(0)|\geq k!\mu_1^{k_0} \text{ for }k_0\geq \frac{n}{\log_{\lambda_1}(\mu_1 r)}.\]
\item Let $\eta>\frac{1}{\mu_1}$. By using the maximum principle for $\phi_n^{'}$, we have the existence of $w_1\in \overline{D(0,\eta)}$ such that $\displaystyle |\phi_n^{'}(w_1)|=\max_{z\in D(0,\eta)}|\phi_n^{'}(w)|$. Further, using Cauchy estimate, we have 
\begin{align*}
k!\mu_1^{k_0}&\leq|\phi_n^{(k_0)}(0)|\\
&=|(\phi_n^{'})^{(k_0-1)}(0)|\\
&<\frac{(k_0-1)!|\phi_n^{'}(w_1)|}{\eta^{k_0-1}}.\\
\end{align*}
This implies that $|\phi_n^{'}(w_1)|>\frac{k_0(\mu_1\eta)^{k_0}}{\eta}\geq C_2\Lambda^n$, where $\Lambda=(\mu_1\eta)^{\frac{1}{\log_{\lambda_1}(\mu_1 r)}}$ and $C_2>0$ is a constant. Note that, here we have used the inequality $k_0\geq \frac{n}{\log_{\lambda_1}(\mu_1 r)} $ to obtain the mentioned value of $\Lambda.$
\item Let $p_1=\phi(w_1)=(z_0,w_0,t_0)$. Since $\Vert\pi_n\Vert=1$, we have 
\[\Vert d_{p_1}(G^n)\Vert\geq C_3\Lambda^n \text{ for some constant }C_3>0. \]
This further implies that 
\[\Vert d_{p_1}(F^{ n})\Vert\geq C_3\Lambda^n.\]
\item Finally, we prove that there exists a point in $\phi(D(0,\eta))$ which is contained in the Julia set of $F$. For this, suppose that $\phi(D(0,\eta))$ is contained in the Fatou component $V$ of $F$. Then, we have $F^{ n}$ converges uniformly to $[1:1:1:0]$ in a neighbourhood $U_2$ of $p_1$.

Now, on $\{z_3\not=0\}$, we have an affine chart of $\mathbb{P}^3$ which is defined as\,-- \\
$\psi\left([z_1:z_2:z_3:t]\right)=\left(\frac{z_1}{z_3},\frac{z_2}{z_3},\frac{t}{z_3}\right)$.
Define $H:\{z_{3}\not= 0\}\cap\{t\not=0\}\to\mathbb{C}^3$ as 
\[H(z_1,z_2,z_3)=\left(\frac{z_1}{z_3}, \frac{z_2}{z_3},\frac{1}{z_3}\right).\]
For sufficiently large $n$, the sequence $H\circ F^n: U_1\to\mathbb{C}^3$ is defined and converges to $(1,1,0)\in\mathbb{C}^3$. Therefore, we have
\begin{equation}\label{9}
d_{(z_n,w_n,t_n)}H\circ d_{((z_0,w_0,t_0)=p_1)}F^n=d_{(z_0,w_0,t_0)}(H\circ F^n)\to 0.
\end{equation}
Note that $(z_n,w_n,t_n)=F^n(z_0,w_0,t_0)$, for $n\in\mathbb{N}$. Also, we have 
\[\Vert(d_{(z_n,w_n,t_n)}H)^{-1}\Vert=1.\]
Finally, 
\[\Vert d_{(z_n,w_n,t_n)}H\circ d_{(z_0,w_0,t_0)}F^n\Vert\geq \frac{\Vert d_{p_1}(F^n)\Vert}{\Vert(d_{(z_n,w_n,t_n)}H)^{-1}\Vert}\geq C_3\Lambda^n\to\infty, \text{ as }n \to\infty,\]
\end{enumerate}
which is a contradiction to \Cref{9}. This completes the proof.
\end{proof}

\section*{Acknowledgements}\vspace{-1mm}
The research of the author was supported by the Institute Post-doctoral Fellowship program at the Institute of Mathematics and Applications, Bhubaneswar, India.

\end{document}